\documentclass[12pt,reqno, oneside]{amsart}
\scrollmode
\usepackage{graphicx}
\usepackage{color}

\usepackage{fullpage}
\usepackage{hyperref}
\newtheorem{thm}{Theorem}[section]
\newtheorem{cor}[thm]{Corollary}
\newtheorem{lem}[thm]{Lemma}

\theoremstyle{definition}
\newtheorem{ex}[thm]{Exercise}

\newcommand{\R}{R}

\newcommand{\1}{\bar 1}

\newcommand{\inv}{^{\langle -1\rangle}}
\newcommand\cf[1]{\mathop{[#1]}}
\newcommand\dpow[2]{\frac{#1^{#2}}{#2!}}
\newcommand\sumz[1]{\sum_{#1=0}^\infty}
\newcommand{\p}{(p-1)}
\newcommand{\psup}[1]{^{(#1)}}

\DeclareMathOperator{\res}{res}

\numberwithin{equation}{subsection}
\numberwithin{thm}{subsection}

\begin{document}
\title{Lagrange Inversion}
\author{Ira M. Gessel$^*$}
\address{Department of Mathematics\\
   Brandeis University\\
   Waltham, MA 02453-2700}
\email{gessel@brandeis.edu}
\begin{abstract}
We give a survey of the Lagrange inversion formula, including different versions and proofs, with applications to combinatorial and formal power series identities.

\end{abstract}

\maketitle

\section{Introduction}
The Lagrange inversion formula is one of the fundamental formulas of combinatorics.  In its simplest form it gives a formula for the power series coefficients of the solution $f(x)$ of the function equation $f(x) = xG(f(x))$ in terms of coefficients of powers of $G$. Functional equations of this form often arise in combinatorics, and our interest is in these applications rather than in other areas of mathematics.

There are many generalizations of Lagrange inversion:  multivariable forms \cite{MR894817}, $q$-analogues \cite{MR638947, MR552269, MR690047, MR938977} noncommutative versions \cite{MR2200854,MR3081645,MR552269,MR2380843} and others
\cite{MR1354969,MR924765,MR927766}.
In this paper we discuss only ordinary one-variable Lagrange inversion, but in greater detail than elsewhere in the literature.

In section \ref{s-formula} we give a thorough discussion of some of the many different forms of Lagrange inversion, prove that they are equivalent to each other, and work through some simple examples involving Catalan and ballot numbers. We address a number of subtle issues that are overlooked in most accounts of Lagrange inversion (and which some readers may want to skip). 
In sections \ref{s-apps} 
we describe applications of Lagrange inversion to identities involving binomial coefficients, Catalan numbers, and their generalizations. In section \ref{s-proofs}, we give several proofs of Lagrange inversion, some of which are combinatorial.

A number of exercises giving  additional results are included.

An excellent introduction to Lagrange inversion can be found in Chapter 5 of Stanley's \emph{Enumerative Combinatorics}, Volume 2. Other expository accounts of Lagrange inversion can be found in Hofbauer \cite{hofbauer}, Bergeron, Labelle, and Leroux \cite[Chapter 3]{MR1629341}, Sokal \cite{MR2529395}, and 	Merlini, Sprugnoli, and Verri \cite{MR2290868}.																																																

\subsection{Formal power series}
Although Lagrange inversion is often presented as a theorem of analysis (see, e.g., Whittaker and Watson \cite[pp.~132--133]{ww}), we will work only with formal power series and formal Laurent series. A good account of  formal power series can be found in  Niven \cite{MR0252386}; we sketch here some of the basic facts. Given a coefficient ring $C$, which for us  will always be an integral domain containing the rational numbers, the ring $C[[x]]$ of \emph{formal power series}
in the variable $x$ with coefficients in $C$ is the set of all ``formal sums" $\sum_{n=0}^\infty c_n x^n$, where $c_n\in C$, with termwise addition and multiplication defined as one would expect using distributivity: $\sum_{n=0}^\infty a_n x^n \cdot \sum_{n=0}^\infty b_n x^n = \sum_{n=0}^\infty c_n x^n$, where $c_n = \sum_{i=0}^n a_i b_{n-i}$. Differentiation of formal power series is also defined termwise.  A series $\sum_{n=0}^\infty c_n x^n$ has a multiplicative inverse if and only if $c_0$ is invertible in $C$.  We may also consider the ring of \emph{formal Laurent series} $C((x))$ whose elements are formal sums $\sum_{n=n_0}^\infty c_n x^n$ for some integer $n_0$, i.e., formal sums $\sum_{n=-\infty}^\infty c_n x^n$ in which only finitely many negative powers of $x$ have nonzero coefficients.
Henceforth will omit the word ``formal" and speak of power series and Laurent series.

We can iterate the power series and Laurent series ring constructions, obtaining, for example the ring $C((x))[[y]]$ of  power series in $y$ whose coefficients are  Laurent series in $x$. 
In any (possibly iterated) power series or Laurent series ring we will say that a set $\{f_\alpha\}$ of series is \emph{summable} if for any monomial $m$ in the variables, the coefficient of $m$ is nonzero in only finitely many $f_\alpha$. In this case the sum $\sum_\alpha f$ is well-defined and we will say that $\sum_\alpha f_\alpha$ is summable.
If we write $\sum_\alpha f_\alpha$ as an iterated sum, then the order of summation is irrelevant. 
If $f(x)=\sum_n c_n x^n$ is a Laurent series in $C((x))$ and $u\in C$, where $C$ may be a power series or Laurent series ring, then we say that that the substitution of $u$ for $x$ is \emph{admissible} if $f(u)=\sum_n c_n u^n$ is summable, and similarly for multivariable substitutions. Admissible substitutions are homomorphisms. If $u$ is a power series or Laurent series $g(x)$ then $f(g(x))$, if summable, is called the \emph{composition} of $f$ and $g$. If $f(x) = c_1x+c_2x^2+\cdots$, where $c_1$ is invertible in $C$, then there is a unique power series $g(x)= c_1^{-1}x +\cdots$ such that $f(g(x)) = x$; this implies that $g(f(x))=x$. We call $g(x)$ the \emph{compositional inverse} of $f(x)$ and write $g(x)= f(x)\inv$. For simplicity, we will always assume that if $f(x) = c_1x+c_2x^2+\cdots$, where $c_1\ne 0$ then $c_1$ is invertible. (Since $C$ is an integral domain, we can always adjoin $c_1^{-1}$ to $C$ if necessary.)

The iterated power series rings $C[[x]][[y]]$ and $C[[y]][[x]]$ are essentially the same, in that both consist of all sums $\sum_{m,n\ge0}c_{m,n}x^m y^n$. We may therefore write $C[[x,y]]$ for either of these rings.  However, the iterated Laurent series rings $C((x))[[y]]$, $C((y))[[x]]$, and $C[[y]]((x))$  are all different: in the first we have
$(x-y)^{-1}=\sum_{n=0}^\infty y^n\!/x^{n+1}$, in the second we have $(x-y)^{-1}=-\sum_{n=0}^\infty x^n\!/y^{n+1}$, and in the third $x-y$ is not invertible. 

It is sometimes convenient   to work with power series in infinitely many variables; for example, we may consider the power series $\sum_{n=0}^\infty r_n t^n$ where the $r_n$ are independent indeterminates. Although we don't give a formal definition of these series, they behave, in our applications, exactly as expected.

We use the notation $\cf{x^n} f(x)$ to denote the coefficient of $x^n$ in the Laurent series $f(x)$. An important fact about the coefficient operator that we will use often, without comment, is that $\cf{x^n} x^k f(x) = \cf{x^{n-k}} f(x)$.

The binomial coefficient $\binom ak$ is defined to be $a(a-1)\cdots (a-k+1)/k!$ if $k$ is a nonnegative integer and 0 otherwise. Thus the binomial theorem $(1+x)^a = \sum_{k=0}^\infty \binom ak x^k$ holds for all $a$.

\section{The Lagrange inversion formula}
\label{s-formula}
\subsection{Forms of Lagrange inversion}
\label{ss-forms}
We will give several proofs of the Lagrange inversion formula in section \ref{s-proofs}. Here we state several different forms of Lagrange inversion and show that they are equivalent.
\begin{thm}
\label{t-1}
Let $\R(t)$ be a  power series not involving $x$. Then there is a unique  power series $f=f(x)$ such that 
$f(x) = x \R(f(x))$, and for any  Laurent series $\phi(t)$ and $\psi(t)$  not involving $x$ and any integer $n$ we have 

{\allowdisplaybreaks
\begin{gather}
\cf{x^n}\phi(f)=\frac{1}{n}\cf{t^{n-1}}\phi'(t)\R(t)^n,\ \text{where } n\ne0,\label{e-a}\\
\cf{x^n}\phi(f)=\cf{t^n}\bigl(1-t\R'(t)/\R(t)\bigr)\phi(t)\R(t)^n, \label{e-b}\\
\phi(f) = \sum_n x^n \cf{t^n} (1-x\R'(t))\phi(t)\R(t)^n, \label{e-c}\\
\cf{x^n}\frac{\psi(f)}{ 1-x\R'(f)}=\cf{t^n}{\psi(t)\R(t)^n}, \label{e-d}\\
\cf{x^n}\frac{\psi(f)}{ 1-f\R'(f)/\R(f)}=\cf{t^n}{\psi(t)\R(t)^n}. \label{e-e}
\end{gather}
}

\end{thm}
We show here these formulas are equivalent in the sense that any one of them is easily derivable from any other;  proofs of these formulas are given in section \ref{s-proofs}. It is clear that \eqref{e-d} and \eqref{e-e} are equivalent since $x=f/R(f)$. Taking $\psi(t)=(1-t\R'(t)/R(t))\phi(t)$ shows that \eqref{e-b} and \eqref{e-e} are equivalent.

To derive \eqref{e-c} from \eqref{e-d}, we rewrite \eqref{e-d} as
\begin{equation}
\label{e-dx}
\frac{\psi(f)}{ 1-x\R'(f)}=\sum_n x^n \cf{t^n}{\psi(t)\R(t)^n}.
\end{equation}
Until now, we have  assumed that  $\phi(t)$ and $\psi(t)$ do not involve $x$. 
We leave it to the reader to see that in \eqref{e-dx} this assumption can be removed. Then \eqref{e-c} follows from 
\eqref{e-dx} by setting $\psi(t) = (1-x\R'(t))\phi(t)$, and similarly \eqref{e-dx} follows from \eqref{e-c}.

Although we allow $R$ to be an arbitrary  power series in Theorem \ref{t-1}, if $R$ has constant term 0 then $f(x)=0$, so we may assume now that $R$ has a nonzero constant term, and thus $f$ and  $R(f)$ are nonzero. Then the equation $f(x) = xR(f(x))$ may be rewritten as $f/R(f) = x$. So if we set $g(t) = t/R(t)$ then we have $g(f) = x$, and thus $g = f\inv$. It is sometimes convenient to rewrite the formulas of Theorem \ref{t-1} using $g$, rather than $R$. Since $1-tR'(t)/R(t) = t g'(t)/g(t)$,  formula  \eqref{e-b}  takes on a slightly simpler form (which will be useful later on) when expressed in terms of $g$ rather than $R$:
\begin{equation}
\label{e-bg}
\cf{x^n}\phi(f)=\cf{t^{n-1}}\phi(t) \frac{g'(t)}{g(t)}\left(\frac{t}{g(t)}\right)^n
 = \cf{t^{-1}}\frac{\phi(t) g'(t)}{g(t)^{n+1}},\end{equation}
For future use, we note also the corresponding form for  \eqref{e-e}:
\begin{equation}
\label{e-eg}
\cf{x^{n-1}}\frac{\psi(f)}{fg'(f)}=\cf{t^n}\psi(t)\left(\frac{t}{g(t)}\right)^n
=\cf{t^0}\frac{\psi(t)}{g(t)^n}.
\end{equation}

To show that \eqref{e-a} and \eqref{e-b} are equivalent, using \eqref{e-bg} in place of \eqref{e-b}, we show that 
\begin{equation}
\label{e-res1}
\cf{t^{-1}} \frac{\phi'(t)}{g(t)^n} = n\cf{t^{-1}}\frac{\phi(t)g'(t)}{g(t)^{n+1}}
\end{equation}
But 
\begin{equation*}
\frac{\phi'(t)}{g(t)^n} - n\frac{\phi(t)g'(t)}{g(t)^{n+1}} = \frac{d\ }{dt} \frac{\phi(t)}{g(t)^n},
\end{equation*}
and the coefficient of $t^{-1}$ in the derivative of any  Laurent series is 0, so \eqref{e-res1} follows.
This shows that \eqref{e-a} and \eqref{e-b} are equivalent if $n\ne 0$. If $\phi$ is a power series, then the coefficient of $x^0$ in $\phi(f)$ is simply the constant term in $\phi$, but if $\phi$ is a more general Laurent series, then the constant term in $\phi(f)$ is not so obvious, and cannot be determined by \eqref{e-a}. In equation \eqref{e-a+} we will give a formula for the constant term in $\phi(f)$ for all $\phi$.

The case $\phi(t) = t^k$ of \eqref{e-a}, with $R(t)=t/g(t)$ may be written
\begin{equation}
\label{e-sj1}
\cf{x^{n}} f^k = \frac{k}{n}\cf{t^{n-k}}\left(\frac{t}{g(t)}\right)^n
  =\frac{k}{n}\cf{t^{-k}}g(t)^{-n}.
\end{equation}
In other words, if $g=f\inv$, and for all integers $k$,
$f^k = \sum_n a_{n,k}x^n$ and $g^k=\sum_n b_{n,k}x^n$ then 
\begin{equation}
\label{e-sj2}
a_{n,k}=\frac kn b_{-k, -n}
\end{equation}
for $n\ne 0$.
Equation \eqref{e-sj2} is known as the \emph{Schur-Jabotinsky theorem}. (See Schur \cite[equation (10)]{MR0011740} and Jabotinsky \cite[Theorem II]{MR0059359}.)

\subsection{Polynomials}
\label{s-polynomials}
We now give a slightly more general form of Lagrange inversion based on the fact that if two polynomials agree at infinitely many values than they are identically equal. This will imply that to prove our Lagrange formulas for all $n$, it is sufficient to prove them in the case in which $n$ is a positive integer. (Some proofs require this restriction.)

By linearity, the formulas of Section \ref{ss-forms} are implied by the special cases in which $\phi(t)$ and $\psi(t)$ are of the form $t^k$ for some integer $k$, and these special cases (especially $k=0$ and $k=1$) are particularly important. These special cases of  \eqref{e-a}, \eqref{e-d}, and \eqref{e-e}  are especially useful and may be written
\begin{gather}
\cf{x^n} f^k = \frac{k}{n} \cf{t^{n-k}} R(t)^n, \ \text{where } n\ne0,\label{e-pa}\\
\cf{x^n}\frac{f^k}{1-xR'(f)}=\cf{t^{n-k}} R(t)^n. \label{e-pd}\\
\cf{x^n}\frac{f^k}{ 1-f\R'(f)/\R(f)}=\cf{t^{n-k}}{\R(t)^n}. \label{e-pe}\end{gather}

In these formulas, let us assume that $R(t)$ has constant term 1. (It is not hard to modify our approach to take care of the more general situation in which the constant term of $R(t)$ is  invertible.) Then the coefficient of $x$ in $f(x)$ is 1, so $f(x)/x$ has constant term 1. If we set $n=m+k$ in \eqref{e-pa}, \eqref{e-pd}, and \eqref{e-pe} then the results may be written 
\begin{gather}
\cf{x^m} (f/x)^k = \frac{k}{m+k} \cf{t^{m}} R(t)^{m+k}, \ \text{where } m+k\ne0,\label{e-qa}\\
\cf{x^m}\frac{(f/x)^k}{1-xR'(f)}=\cf{t^{m}} R(t)^{m+k} \label{e-qd}\\
\cf{x^m}\frac{(f/x)^k}{1-fR'(f)/R(f)}=\cf{t^{m}} R(t)^{m+k}. \label{e-qe}
\end{gather}
It is easy to see that in each of these equations, for fixed $m$ both sides are polynomials in $k$. Thus if these equalities hold whenever $k$ is a positive integer, then they hold as identities of polynomials in $k$.  Moreover, although \eqref{e-qa} is invalid for $k=-m$,  if $m>0$ we may take the limit as $k\to -m$ with l'H\^opital's rule to obtain 
\begin{equation}
\label{e-log1}
\cf{x^m} (f/x)^{-m}=\cf{x^0} f^{-m}= - m\cf{t^m}\log R. 
\end{equation}
(Note that \eqref{e-log1} does not hold for $m=0$.) By linearity, \eqref{e-log1} yields a supplement to \eqref{e-a} that takes care of the case $n=0$:
\begin{equation}
\label{e-a+}
\cf{x^0} \phi(f)=\cf{t^0} \phi(t) + \cf{t^{-1}} \phi'(t) \log R.
\end{equation}

We can also differentiate \eqref{e-qa} with respect to $k$ and then set $k=0$ to obtain
\begin{equation}
\label{e-log2}
[x^m] \log (f/x) = \frac{1}{m} \cf{t^m} R(t)^m, \ \text{for $m\ne 0$}.
\end{equation}

Returning to \eqref{e-pa}--\eqref{e-pe}, we see that if they hold when $n$ and $k$ are positive integers, then they also hold when $n$ and $k$ are arbitrary integers. (Note that if $n<k$ then everything is zero.)

\subsection{A simple example: Catalan numbers}
The Catalan numbers $C_n$ may defined by the equation 
\begin{equation}
c(x) = 1+ xc(x)^2 \label{e-cat}
\end{equation}
for their generating function $c(x) = \sum_{n=0}^\infty C_n x^n.$
The quadratic equation \eqref{e-cat} has two solutions, $\bigl(1\pm\sqrt{1-4x}\bigr)/(2x)$, but only the minus sign gives a power series, so 
\begin{equation*}
c(x) = \frac{1-\sqrt{1-4x}}{2x}.
\end{equation*}

Unfortunately \eqref{e-cat} is not of the form $f(x) = xR(f(x))$, so we cannot apply directly any of the versions of Lagrange inversion that we have seen so far.

One way to apply Lagrange inversion is to set $f(x) = c(x) -1$, so that $f=x(1+f)^2$.
We may then apply Theorem \ref{t-1} to the case $R(t) = (1+t)^2$. The equation $f=x(1+f)^2$ has the solution 
\begin{equation*}
f(x) = c(x) -1 =xc(x)^2= \frac{1-\sqrt{1-4x}}{2x} -1.
\end{equation*}
Then \eqref{e-a} with $\phi(t) = (1+t)^k$ gives for $n>0$
\begin{align*}
\cf{x^n} c(x)^k &= \cf{x^n} (1+f)^k = \frac1n \cf{t^{n-1}} k(1+t)^{k-1}(1+t)^{2n}\\
  &=\frac kn \cf{t^{n-1}}\binom{2n+k-1}{n-1}.
\end{align*}
Thus since the constant term in $c(x)^k$ is 1, we have
\begin{equation*}
c(x)^k = 1+\sum_{n=1}^\infty \frac{k}{n}\binom{2n+k-1}{n-1}x^n.
\end{equation*}
The sum may also be written
\begin{equation}
\label{e-balalt}
c(x)^k=\sum_{n=0}^\infty\frac{k}{2n+k}\binom{2n+k}{n} x^n
  = \sum_{n=0}^\infty \frac{k}{n+k}\binom{2n+k-1}{n} x^n.
\end{equation}
These formulas are valid for all $k$  
except where  $n=-k/2$ in the first sum in \eqref{e-balalt} or $n=-k$ in the second sum in  \eqref{e-balalt}.
These coefficients are called \emph{ballot numbers}, and for $k=1$ \eqref{e-balalt} gives the usual formula for the Catalan numbers, $C_n = \frac{1}{n+1} \binom{2n}{n}$. 

Equation \eqref{e-b} with $\phi(t) = (1+t)^k$ gives a formula for the ballot number as a difference of two binomial coefficients,
\begin{align*}
\cf{x^n} c(x)^k &= \cf{t^n}\frac{1-t}{1+t} (1+t)^k (1+t)^{2n}\\
  &=\cf{t^n} (1-t) (1+t)^{2n+k-1}\\
  &=\binom{2n+k-1}{n} -\binom{2n+k-1}{n-1},
\end{align*}
and
equation \eqref{e-c} with $\phi(t) = (1+t)^k$ gives another such formula,
\begin{align*}
c(x)^k   &=\sum_n x^n \cf{t^n}\bigl((1+t)^{2n+k}-2x(1+t)^{2n+k+1}\bigr)\\
  &=\sum_n x^n \left[ \binom{2n+k}{n} -2 \binom{2n+k-1}{n-1}\right].
\end{align*}

Finally, \eqref{e-log2} gives
\begin{equation*}
\cf{x^m} \log (f/x) = \cf{x^m} 2\log c(x) = \frac 1m \cf{t^m}(1+t)^{2m}=\frac1m \binom{2m}{m},
\end{equation*}
so
\begin{equation*}
\log c(x) = \sum_{m=1}^\infty \frac{1}{2m}\binom{2m}{m} x^m.
\end{equation*}

Since $R(t)=(1+t)^2$, we have $R'(t) = 2(1+t)$, so $1-xR'(f) = 1-2xc(x) = \sqrt{1-4x}$. Thus \eqref{e-d}, with $\psi(t) = (1+t)^k$ gives
\begin{equation*}
\cf{x^n}\frac{c(x)^k}{\sqrt{1-4x}}=\cf{t^n}(1+t)^k(1+t)^{2n}=\binom{2n+k}{k},
\end{equation*}
so
\begin{equation}
\label{e-midbin}
\frac{c(x)^k}{\sqrt{1-4x}}=\sumz n \binom{2n+k}{k} x^n.
\end{equation}

Equating coefficients of $x^n$ in $c(x)^k c(x)^l=c(x)^{k+l}$, and using \eqref{e-balalt}, gives the convolution identity
\begin{equation*}
\sum_{i+j=n} \frac{k}{2i+k}\binom{2i+k}{i}\cdot\frac{l}{2j+l}\binom{2j+l}{j}=\frac{k+l}{2n+k+l}\binom{2n+k+l}{n}.
\end{equation*}

Similarly using \eqref{e-balalt} and \eqref{e-midbin} we get
\begin{equation*}
\sum_{i+j=n} \frac{k}{2i+k}\binom{2i+k}{i}\binom{2j+l}{j}=\binom{2n+k+l}{n}.\end{equation*}
These convolution identities are special cases of identities discussed in section \ref{ss-fc}.

\begin{ex} Derive these formulas for $c(x)$ in other ways by applying Lagrange inversion to the equations  $f=x/(1-f)$ and $f=x(1+f^2)$.
\end{ex}

\subsection{A generalization}
There is another way to apply Lagrange inversion to the equation $c(x) = 1+x c(x)^2$ that, while very simple, has far-reaching consequences. Consider the equation
\begin{equation}
\label{e-cat1}
F = z(1+xF^2)
\end{equation}
where we think of $F$ as a power series in $z$ with coefficients that are polynomials in $x$.
We may apply \eqref{e-pa} to \eqref{e-cat1} to get
\begin{equation*}
\cf{z^n} F^k = \frac{k}{n}\cf{t^{n-k}}(1+xt^2)^n.
\end{equation*}
The right side is 0 unless $n-k$ is even, and for $n=2m+k$ we have
\begin{equation*}
\cf{z^{2m+k}} F^k = \frac{k}{2m+k}\cf{t^{2m}}(1+xt^2)^{2m+k}
  =\frac{k}{2m+k} \binom{2m+k}{m}x^m.
\end{equation*}
Thus we have (if $k$ is an integer but not a negative even integer)
\begin{equation}
\label{e-F1}
F^k = \sum_{m=0}^\infty \frac{k}{2m+k}\binom{2m+k}{m}x^m z^{2m+k}.
\end{equation}
Now let $c$ be the result of setting $z=1$ in $F$ (an admissible substitution), so by \eqref{e-F1}, we have
\begin{equation*}
c^k = \sum_{m=0}^\infty \frac{k}{2m+k}\binom{2m+k}{m}x^m 
\end{equation*}
and by \eqref{e-cat1} we have $c=1+xc^2$. Moreover, as we have seen before, $c=1+xc^2$ has a unique  power series solution.

The same idea works much more generally, but we must take care that the substitution is admissible. For example, we can solve $f=x(1+f)$ by Lagrange inversion, but we cannot set $x=1$ in the solution.

The case in which the coefficients of $R(t)$ are indeterminates is easy to deal with. 
\begin{thm} 
\label{t-indets}
Suppose that $R(t) = \sum_{n=0}^\infty r_n t^n$, where the $r_n$ are indeterminates. Then there is a unique power series $f$ satisfying $f=R(f)$. 
If $\phi(t)$ is a power series then 
\begin{equation}
\label{e-r00}
\phi(f) = \phi(0)+ \sum_{n=1}^\infty \frac{1}{n}\cf{t^{n-1}}\phi'(t) R(t)^n,
\end{equation}
and for any  Laurent series $\phi(t)$ and $\psi(t)$ we have
\begin{gather}
\phi(f) = \cf{t^0}\phi(t) +\cf{t^{-1}} \phi'(t) \log (R/r_0) + \sum_{n\ne 0} \frac{1}{n}\cf{t^{n-1}}\phi'(t) R(t)^n,\label{e-r0}\\
\phi(f) = \sum_n \cf{t^n}\bigl(1-t\R'(t)/\R(t)\bigr)\phi(t)\R(t)^n
    = \sum_n \cf{t^n} (1-\R'(t))\phi(t)\R(t)^n,\label{e-r1}\\
\frac{\psi(f)}{ 1-\R'(f)}
=\frac{\psi(f)}{ 1-f\R'(f)/\R(f)}
=\sum_n\cf{t^n}{\psi(t)\R(t)^n}.\label{e-r2}
\end{gather}
In \eqref{e-r1} and \eqref{e-r2} the sum is over all integers $n$.
\end{thm}
\begin{proof}
These formulas follow from equations \eqref{e-a} to \eqref{e-e} on making the admissible substitution $x=1$, where for \eqref{e-r0} we have included the correction term given by \eqref{e-a+}, modified to take into account that the constant term of $R(t)$ is $r_0$ rather than 1. Uniqueness follows by equating coefficients of the monomials $r_0^{i_0} r_1^{i_1}\cdots$ on both sides of $f=R(f)$, which gives a recurrence that determines them uniquely.\end{proof}

We would like to relax the requirement  in Theorem \ref{t-indets} that the $r_n$ be indeterminates. To do this, we can take any of the formulas of Theorem \ref{t-indets} and apply any admissible substitution for the $r_n$.
For example, the following result, while not the most general possible, is sometimes useful. 
\begin{thm}
\label{t-genli}
Suppose that $R(t) = \sumz n r_n t^n$, where the coefficients lie in some power series ring $C[[u_1, u_2,\dots]]$, and that each $r_n$ with $n>0$ is divisible by some $u_i$. Then there is a unique power series $f$ satisfying $f=R(f)$, and formulas \eqref{e-r00} to \eqref{e-r2} hold.\qed
\end{thm}
We note that more generally, any admissible substitution will yield a solution of $f=R(f)$ to which these formulas hold, but uniqueness is not guaranteed. For example, the equation $f=x+yf^2$ has the unique  power series solution,
\begin{equation*}
f = \frac{1-\sqrt{1-4xy}}{2y}=\sum_{n=0}^\infty \frac{1}{n+1}\binom{2n}{n} x^{n+1}y^n.
\end{equation*}
The admissible substitution $y=1$ gives  $f=\frac12(1-\sqrt{1-4x})$ as a  power series solution of the equation $f=x+f^2$. However, the equation $f=x+f^2$ has another  power series solution, 
\begin{math}
f=\frac12(1+\sqrt{1-4x}).
\end{math}

\begin{ex}
The equation $f=x+f^2$ has two  power series solutions, $f=\frac12(1\pm\sqrt{1-4x})$. However, according to Theorem \ref{t-1}, the equivalent equation $f=x/(1-f)$ has only one  power series solution. Explain the discrepancy.
\end{ex}

\subsection{Explicit formulas for the coefficients}
It is sometimes useful to have an explicit formula for the coefficients of $f^k$ where $f = xR(f)$. With $R(t)=\sum_{n=0}^\infty r_n t^n$, if we expand $R(t)^n$ by the multinomial theorem then \eqref{e-pa} gives
\begin{equation}
\label{e-lagr2}
f^k= \sum_{\substack{n_0+n_1+\cdots = n\\
n_1+2n_2+3n_3+\cdots=n-k}}
k\frac{(n-1)!}{n_0!\,n_1!\dots}r_0^{n_0}r_1^{n_1}\cdots x^n.
\end{equation}

We might also want to express the coefficients of $f^k$ in terms of the coefficients of $g=f\inv$.
Suppose that $g(x) = x - g_2 x^2 - g_3 x^3 -\cdots$, where the minus signs and the assumption that the coefficient of $x$ in $g$ is 1 make our formula simpler with no real loss of generality.  Then \eqref{e-pa} gives
\begin{equation*}
[x^m] f^k = \frac km [t^{m-k}] \left(\frac{1}{1-g_2t-g_3t^2-\cdots}\right)^m.
\end{equation*}
Expanding  by the binomial theorem and simplifying gives 
\begin{equation*}
f^k = \sum_{\substack{n_2+n_3+\cdots=n-m\\
n_2+2n_3+\cdots=m-k
}}
k\frac{(n-1)!}{m!\, n_2!\,n_3!\cdots}g_2^{n_2}g_3^{n_3}\cdots x^m
\end{equation*}
where the sum is over all $m$, $n$, and $n_2,n_3,\dots$ satisfying the two subscripted equalities. If we replace $m$ with $n_0$ then we may write the formula as 
\begin{equation}
\label{e-lagr5}
f^k = \sum_{\substack{n_0+n_2+n_3+\cdots=n\\
2n_2+3n_3+\cdots=n-k
}}
k\frac{(n-1)!}{n_0!\, n_2!\,n_3!\cdots}x^{n_0}g_2^{n_2}g_3^{n_3}\cdots 
\end{equation}
and we see that the coefficients here are exactly the same as the coefficients in \eqref{e-lagr2} (with $n_1=0$). 

\begin{ex} 
Explain the connection between \eqref{e-lagr2} and \eqref{e-lagr5} without using Lagrange inversion.
\end{ex}

\subsection{Derivative formulas}

Lagrange inversion, especially in its analytic formulations, is often stated in terms of derivatives. We give here several derivative forms of Lagrange inversion.
\begin{thm}
\label{t-der}
Let $G(t) = \sum_{n=0}^\infty g_n t^n$, where the $g_i$ are indeterminates. Then there is a unique power series  $f$ satisfying 
\begin{equation*}
f = x + G(f)
\end{equation*}
and for any   power series $\phi(t)$ and $\psi(t)$ we have

\begin{equation}
\phi(f)=\sum_{m=0}^\infty \frac{d^m\ }{dx^m}
  \left(\phi(x)\bigl(1-G'(x)\bigr)\frac{G(x)^m}{ m!}\right), \label{e-d1}
\end{equation}
\begin{equation}
\phi(f) = \phi(x) +\sum_{m=1}^\infty  \frac{d^{m-1}\ }{dx^{m-1}}
 \left( \phi'(x) \frac{G(x)^m}{ m!}\right),\label{e-d2}
\end{equation}
and
\begin{equation}
\frac{\psi(f)}{ 1-G'(f)}=\sum_{m=0}^\infty \frac{d^m\ }{dx^m}\left(\psi(x)\frac{G(x)^m}{ m!}\right). \label{e-d3}
\end{equation}
\end{thm}
\begin{proof}
We first prove \eqref{e-d3}. By \eqref{e-r2} we have
\begin{equation*} 
\frac{\psi(f)}{ 1-G'(f)}=\sum_{n=0}^\infty \cf{t^n} \psi(t) (x+G(t))^n = \sum_{m=0}^\infty \sum_{n=0}^\infty\cf{t^n}  \binom nm x^{n-m}\psi(t)G(t)^m.
\end{equation*}

So to prove \eqref{e-d3}, it suffices to prove that 
\begin{equation*}
\frac{d^{m}\ }{dx^{m}}
\left( \psi(x) \frac{G(x)^m}{ m!}\right) = \sum_{n=0}^\infty\cf{t^n} \binom nm x^{n-m}\psi(t)G(t)^m.
\end{equation*}
We show more generally, that for any power series $\alpha(t)$ we have 
\begin{equation}
\label{e-diffa}
\frac{d^{m}\ }{dx^{m}}\frac{\alpha(x)}{m!}
 = \sum_{n=0}^\infty\cf{t^n} \binom nm x^{n-m}\alpha(t).
\end{equation}
But by linearity, it is sufficient to prove \eqref{e-diffa} for the case $\alpha(t) = t^j$, where both sides are equal to 
$\binom{j}{m} x^{j-m}$.

Next, \eqref{e-d1} follows from \eqref{e-d3} by taking $\psi(t) = (1-G'(t))\phi(t)$.

Finally, we derive \eqref{e-d2} from \eqref{e-d1}. Writing $D$ for $d/dx$, $\phi$ for $\phi(x)$ and $G$ for $G(x)$, we have 
\begin{equation*}
D^m \left(\phi \frac{ G^m}{m!}\right) = D^{m-1}\left(\frac{D (\phi G^m)}{m!}\right)
   = D^{m-1}\left(\phi'  \frac{G^m}{m!} + \phi G'\frac{G^{m-1}}{(m-1)!}\right).
\end{equation*}
Thus 
\begin{equation*}
D^{m-1}\left(\phi'  \frac{G^m}{m!}\right) = D^m \left(\phi \frac{ G^m}{m!}\right) -
  D^{m-1}\left( \phi G'\frac{G^{m-1}}{(m-1)!}\right),
\end{equation*}
so
\begin{equation*}
\phi+\sum_{m=1}^\infty D^{m-1}\left(\phi'  \frac{G^m}{m!}\right) 
  = \sum_{m=0}^\infty D^m \left(\phi \frac{ G^m}{m!}\right) -\sum_{m=0}^\infty  D^{m}\left( \phi G'\frac{G^{m}}{m!}\right).\qedhere
\end{equation*}

\end{proof}
As before, applying an admissible substitution allows more general coefficients to be used.

As an application of these formulas, let us take $G(x)=z H(x)$ and consider the formula
\begin{equation*}
\phi(f) \cdot \frac{\psi(f)}{1-zH'(f)} =  \frac{\phi(f)\psi(f)}{1-zH'(f)}.
\end{equation*}
Applying \eqref{e-d2} and \eqref{e-d3} to the left side and \eqref{e-d3} to the right, and then equating coefficients of $z^n$ gives the convolution identity
\begin{equation}
\label{e-cauchy1}
\sum_{m=0}^n \binom{n}{m} \frac{d^{m-1}\ }{dx^{m-1}}
 \left( \phi'(x)H(x)^m\right)
 \frac{d^{n-m}\ }{dx^{n-m}}\left(\psi(x)H(x)^{n-m}\right)
  = \frac{d^{n}\ }{dx^{n}}\bigl(\phi(x)\psi(x) H(x)^{n}\bigr)
\end{equation}
and similarly, expanding $\phi(f)\psi(f)$ in two ways using \eqref{e-d2} gives
\begin{equation}
\label{e-cauchy2}
\sum_{m=0}^n \binom{n}{m} \frac{d^{m-1}\ }{dx^{m-1}}
 \left( \phi'(x)H(x)^m\right)
 \frac{d^{n-m-1}\ }{dx^{n-m-1}}\left(\psi'(x)H(x)^{n-m}\right)
  = \frac{d^{n-1}\ }{dx^{n-1}}\bigl((\phi(x)\psi(x))' H(x)^{n}\bigr).
\end{equation}
Here $d^{m-1}\phi'(x)/dx^{m-1}$ for $m=0$ is to be interpreted as $\phi(x)$.
Formulas \eqref{e-cauchy1} and \eqref{e-cauchy2} were found by Cauchy \cite{cauchy1826}; a formula equivalent to \eqref{e-cauchy1} had been found earlier by 
Pfaff \cite{pfaff}. A detailed historical discussion of these identities and generalizations has been given by Johnson \cite{MR2281160}; see also Chu
\cite{MR2604484,MR3060871} and Abel \cite{MR3139570}. Our approach to these identities has also been given by Huang and Ma \cite{MR3365045}.
\begin{ex}
With the notation of Theorem \ref{t-der}, show that
for any positive integer $k$, 
\[G(f)^k=\sum_{m=0}^\infty \frac{k}{ (m+k)\,m!}\frac{d^m\ }{
dx^m} G(x)^{m+k}.\]
\end{ex}

\section{Applications}
\label{s-apps}
In this section we describe some applications of Lagrange inversion.

\subsection{A rational function expansion}
It is surprising that Lagrange inversion can give interesting results when the solution to the equation to be solved is rational. 

We consider the equation
\begin{equation*}
f=1+a+abf,
\end{equation*}
with solution
\begin{equation*}
f=\frac{1+a}{1-ab}.
\end{equation*}
We apply \eqref{e-r2} with $R(t) = 1+a+abt$ and  $\psi(t) = t^r (1+bt)^s$. Here we have
$1+bf = (1+b)/(1-ab)$ and $1-R'(f) = 1-ab$. Then
\begin{align*}
\frac{(1+a)^r(1+b)^s}{(1-ab)^{r+s+1}}
  &=\sum_n [t^n] t^r (1+bt)^s (1+a+abt)^n\\
  &=\sum_{n,i} [t^n]\binom{n}{i}a^i t^r(1+bt)^{s+i}\\
  &=\sum_{n,i,j} [t^n] \binom ni a^i \binom{s+i}j b^j t^{r+j}\\
  &=\sum_{i,j}\binom{r+j}i\binom{s+i}j a^i b^j.
\end{align*}
For another approach to this identity, see Gessel and Stanton \cite{MR773560}.

\subsection{The tree function}
\label{ss-tree}
In applying Lagrange inversion, the nicest examples are those in which the series $R(t)$ has the property that there is a simple formula for the coefficients of $R(t)^n$, and these simple formulas usually come from the exponential function or the binomial theorem. In this section we discuss the simplest case, in which $R(t)=e^t$. Later, in section \ref{s-raney}, we discuss a more complicated example involving the exponential function.

Let $T(x)$ be the  power series satisfying 
\begin{equation*}
T(x) = xe^{T(x)}.
\end{equation*}
Equivalently, $T(x)=(xe^{-x})^{\langle -1\rangle}$ and thus $T(xe^{-x})=x$.
Then by the properties of exponential generating functions (see, e.g., Stanley \cite[Chapter 5]{MR1676282}), $T(x)$ is the exponential generating function for rooted trees and $e^{T(x)}=T(x)/x$ is the exponential generating function for forests of rooted trees. We shall call $T(x)$ the \emph{tree function} and we shall call $F(x) = e^{T(x)}$ the \emph{forest function} The tree function is closely related to the Lambert $W$ function \cite{MR1414285,MR1809988} which may be defined by $W(x) = -T(-x)$. Although the Lambert $W$ function is better known, we will state our results in terms of the tree and forest functions.

Applying \eqref{e-pa}  with $R(t)=e^t$ gives
\begin{equation}
\label{e-Tseries}
T(x)=\sum_{n=1}^\infty n^{n-1}\frac{x^n}{n!}
\end{equation}
and more generally,
\begin{equation}
\label{e-krooted}
\frac{T(x)^k}{k!} = \sum_{n=k}^\infty kn^{n-k-1}\binom nk \frac{x^n}{n!}
\end{equation}
for all positive integers $k$, and
\begin{equation}
\label{e-kforest}
F(x)^k = e^{kT(x)}=\sum_{n=0}^\infty k(n+k)^{n-1} \frac{x^n}{n!}
\end{equation}
for all $k$.
Equation \eqref{e-krooted} implies that there are $kn^{n-k-1}\binom nk$ forests of $n$ rooted trees on $n$ vertices  and equation \eqref{e-kforest} implies that there are $k(n+k)^{n-1}$ forests with vertex set $\{1,2,\dots,n+k\}$  in which the roots are 1, 2, \dots, $k$.

An interesting special case of \eqref{e-kforest} is $k=-1$, which may be rearranged to 
\begin{equation}
\label{e-primeparking}
F(x) = \biggl(1-\sum_{n=1}^\infty (n-1)^{n-1}\dpow xn\biggr)^{-1}.
\end{equation}
Equation \eqref{e-primeparking} may be interpreted in terms of prime parking functions \cite[Exercise 5.49f, p.~95; Solution, p.~141]{MR1676282}.

Applying \eqref{e-qe} gives
\begin{equation}
\label{e-f1}
\frac{F(x)^k}{1-T(x)} = \sum_{n=0}^\infty (n+k)^n \dpow xn.
\end{equation}

A.~Lacasse \cite[p.~90]{lacasse} conjectured an identity that may be written as 
\begin{equation}
\label{e-lacasse}
U(x)^3 - U(x)^2=\sum_{n=0}^\infty n^{n+1}\dpow xn,
\end{equation}
where 
\begin{equation*}
U(x) = \sum_{n=0}^\infty n^n \dpow xn.
\end{equation*}
Proofs of Lacasse's conjecture were given by Chen et al. \cite{MR3066349}, Prodinger \cite{MR3089711}, Sun \cite{MR3066350}, and Younsi \cite{younsi}.
We will prove \eqref{e-lacasse} by showing that both sides are equal to $T(x)/(1-T(x))^3$.

To do this, we first note that the right side of \eqref{e-lacasse} is 
\begin{equation}
\label{e-la1}
\frac{d\ }{dx}\sum_{n=0}^\infty (n-1)^n \dpow xn = 
  \frac{d\ }{dx} \frac{e^{-T(x)}}{1-T(x)} = \frac{e^{-T(x)}T(x) T'(x)}{(1-T(x))^2}.
\end{equation}
Differentiating \eqref{e-Tseries} with respect to $x$ gives
\begin{equation*}
T'(x) = \sum_{n=0}^\infty (n+1)^n\dpow xn,
\end{equation*}
which by \eqref{e-f1} is equal to $e^{T(x)}/(1-T(x))$. Thus \eqref{e-la1} is equal to 
$T(x)/(1-T(x))^3$.
But by \eqref{e-f1}, $U(x) = 1/(1-T(x))$ from which it follows easily that $U(x)^3 - U(x)^2 = T(x)/(1-T(x))^3$.

The series $T(x)$ and $U(x)$ were studied by Zvonkine \cite{zvonkine}, who showed that 
$D^k U(x)$ and $D^{k} U(x)^2$, where $D=x d/dx$, are polynomials in $U(x)$.

The series $\sum_{n=0}^\infty (n+k)^{n+m} x^n\!/n!$, where $m$ is an arbitrary integer, can be expressed in terms of $T(x)$. (The case $k=0$ has been studied by Smiley \cite{smiley}.)
We first deal with the case in which $m$ is negative.

\begin{thm} 
\label{t-mneg}
Let $l$ be a positive integer. Then for some polynomial $p_l(u)$ of degree $l-1$, with coefficients that are rational functions of $k$, we have
\begin{equation}
\label{e-mneg}
\sum_{n=0}^\infty (n+k)^{n-l} \dpow xn = e^{kT(x)} p_l(T(x)).
\end{equation}
The first three polynomials $p_l(u)$ are 
\begin{align*}
p_1(u) &=\frac{1}{k},\\
p_2(u) &=\frac{1}{k^2}-\frac{u}{k(k+1)},\\
p_3(u) &=\frac{1}{k^3}-\frac{(2k+1)}{k^2(k+1)^2}u+ \frac{u^2}{k(k+1)(k+2)}.
\end{align*}
\end{thm}

Before proving Theorem \ref{t-mneg} let us check \eqref{e-mneg} for $l=1$ and $l=2$.
The case $l=1$ is equivalent to \eqref{e-kforest}. For $l=2$, we have
\begin{equation*}
e^{kT(x)}=F(x)^k = k\sumz n (n+k)^{n-1}\dpow xn
\end{equation*}
and 
\begin{align*}
e^{kT(x)}T(x)&=F(x)^k\cdot xF(x) = xF(x)^{k+1}\\
  &=x(k+1)\sumz n (n+k+1)^{n-1}\dpow xn\\
  &= (k+1)\sum_{n=0}^\infty n(n+k)^{n-2}\dpow xn.
\end{align*}
Thus
\begin{align*}
e^{kT(x)}\left(\frac{1}{k^2} - \frac{T(x)}{k(k+1)}\right)
   &=\frac 1k \sumz n \bigl((n+k)^{n-1} - n(n+k)^{n-2}\bigr)\dpow xn\\
   &=\sumz n (n+k)^{n-2}\dpow xn.
 \end{align*}
 
The general case can be proved in a similar way. 
(See Exercise \ref{ex-mneg}.) 
However, it is instructive to take a different approach, using finite differences 
(cf.~Gould \cite{MR0480057}), 
that we we will use again in section \ref{ss-fc}.

Let $s$ be a function defined on the nonnegative integers. The \emph{shift operator} $E$ takes $s$ to the function $Es$ defined by $(Es)(n) = s(n+1)$. We denote by $I$ the \emph{identity operator} that takes $s$ to itself and by $\Delta$ the \emph{difference operator} $E-I$, so $(\Delta s)(n) = s(n+1)-s(n)$. It is easily verified that if $s$ is a polynomial of degree $d>0$ with leading coefficient $L$ then $\Delta s$ is a polynomial of degree $d-1$ with leading coefficient $dL$, and if $s$ is a constant then $\Delta s=0$. 
The \emph{$k$th difference} of $s$ is the function $\Delta^k s$. 
Thus if $s$ is a polynomial of degree $d$ with leading coefficient $L$ then  $\Delta^k s=0$ for $k>d$ and $\Delta^d s$ is the constant $d!\,L$.

Since the operators $E$ and $I$ commute, we can expand $\Delta^k=(E-I)^k$ by the binomial theorem to obtain
\begin{equation*}
(\Delta^k s)(n) = \sum_{i=0}^k (-1)^{k-i}\binom{k}{i} s(n+i).
\end{equation*}
We may summarize the result of this discussion in the following lemma.

\begin{lem} Let $s$ be a polynomial of degree $d$ with leading coefficient $L$. Then 
\[\sum_{i=0}^k (-1)^{k-i}\binom{k}{i} s(n+i)\] is 0 if $k>d$ and is the constant $d!\,L$ for $k=d$.
\qed
\end{lem}

\begin{proof}[Proof of Theorem  \ref{t-mneg}]
If we set $x=ue^{-u}$ in Theorem \ref{t-mneg} and use the fact that $T(ue^{-u})=u$ we see that Theorem \eqref{t-mneg} is equivalent to the formula
\begin{equation*}
e^{-ku}\sumz n (n+k)^{n-l}\dpow{(ue^{-u})}{n} =p_l(u).
\end{equation*}
We have
\begin{equation}
\label{e-ab1}
\begin{aligned}
e^{-ku}\sumz n (n+k)^{n+m}\dpow{(ue^{-u})}{n}
  &=\sumz n (n+k)^{n+m}\dpow un e^{-(n+k)u}\\
  &=\sumz j \dpow uj\sum_{n=0}^j (-1)^{j-n} \binom jn (n+k)^{j+m}
\end{aligned}
\end{equation}
If $m=-l$ is a negative integer and $j\ge l$ then $(n+k)^{j+m}=(n+k)^{j-l}$ is a polynomial in $n$ of degree less than $j$, so the inner sum in \eqref{e-ab1} is 0. 
Thus \eqref{e-mneg} follows, with 
\begin{equation*}
p_l(u) = \sum_{j=0}^{l-1} \dpow{u}{j} \sum_{n=0}^j  (-1)^{j-n} \binom jn (n+k)^{-(l-j)}.\qedhere
\end{equation*}
\end{proof}

We cannot set $k=0$ in \eqref{e-mneg}, since the $n=0$ term on the left is $k^{-l}$. However it is not hard to evaluate $\sum_{n=1}^\infty n^{n-l} x^n\!/n!$.

\begin{thm}
Let $q_l(u)$ be the result of setting $k=1$ in $p_l(u)$. Then 
\begin{equation*}
\sum_{n=1}^\infty n^{n-l} \dpow xn = T(x) q_l(T(x)).
\end{equation*}
\end{thm}
\begin{proof}
We have 
\begin{equation*}
\sum_{n=1}^\infty n^{n-l} \dpow xn  = x\sum_{n=0}^\infty (n+1)^{n-l} \dpow xn
\end{equation*}
By \eqref{e-mneg} with $k=1$ this is 
\begin{equation*}
xe^{T(x)}q_l(T(x)) = T(x)q_l(T(x)).
\end{equation*}
\end{proof}

\begin{ex}
\label{ex-mneg}
Prove Theorem \ref{t-mneg} by finding a formula for the coefficient $t_j(n)$ of $x^n\!/n!$ in $e^{kT(x)}T(x)^j$ and showing that 
$(n+k)^{n-l}$ can be expressed as a linear combination, with coefficients that are rational functions of $k$, of $t_0(n),\dots, 
t_{l-1}(n)$.
\end{ex}

Next we consider  $\sum_{n=0}^\infty (n+k)^{n+m} x^n\!/n!$ where $m$ is a nonnegative integer. (We evaluated the case $k=0, m=1$ in our discussion of Lacasse's conjecture.)

\begin{thm}
\label{t-mpos}
Let $m$ be a nonnegative integer. Then there exists a polynomial $r_m(u,k)$, with integer coefficients, of degree $m$ in $u$ and degree $m$ in $k$, such that
\begin{equation*}
\sum_{n=0}^\infty (n+k)^{n+m} \dpow xn = e^{kT(x)}\frac{r_m(T(x), k)}{(1-T(x))^{2m+1}}.
\end{equation*}
The first three polynomials $r_m(u,k)$ are
\begin{align*}
r_0(u,k)&=1,\\
r_1(u,k)&=k + (1-k)u,\\
r_2(u,k)&=k^2 + (1+3k-2k^2)u+ (2-3k+k^2)u^2.
\end{align*}
\end{thm}

\begin{proof}[Proof sketch]
We give here a sketch of a proof that tells us something interesting about the polynomials $r_m(u,k)$; for a more direct approach see Exercise \ref{ex-rm}.

As in the proof of Theorem \ref{t-mneg} we set $x=ue^{-u}$ and consider the sum 
on the left side of \eqref{e-ab1}.

Following Carlitz \cite{MR570168}, we define the \emph{weighted Stirling numbers of the second kind} $R(n,j,k)$ by 
\begin{equation*}
R(n,j,k) = \frac{1}{j!}\sum_{i=0}^j (-1)^{j-i}\binom ji (k+i)^n,
\end{equation*}
so that 
\begin{equation}
\label{e-ws}
\sum_{n=0}^\infty R(n,j,k) \dpow xn = e^{k x}\frac{(e^x-1)^j}{j!}.
\end{equation}
(For $k=0$, $R(n,j,k)$ reduces to the ordinary Stirling number of the second kind $S(n,j)$.)
Equation \eqref{e-ws} implies that the $R(n,j,k)$ is a polynomial in $k$ with integer coefficients.
Then  \eqref{e-ab1} is equal to 
\begin{math}
\sum_{j=0}^\infty  R(j+m,j,k) u^j.
\end{math}
It is not hard to show that for fixed $m$, $R(j+m, j,k)$ is a polynomial in $j$ of degree $2m$.  

Thus
\begin{equation*}
\sum_{j=0}^\infty  R(j+m,j,k) u^j= \frac{r_m(u,k)}{(1-u)^{2m+1}}
\end{equation*}
for some polynomial $r_m(u,k)$ of degree at most $2m$. 

We omit the proof that $r_m(u,k)$ actually has degree $m$ in $u$.
\end{proof}

For $k=1$, the coefficients of $r_m(u,1)$ are positive integers, sometimes called 
\emph{second-order Eulerian numbers}; see, for example, \cite[p.~270]{MR1397498} and \cite{MR0462961}.  

\begin{ex}
\label{ex-rm} Give an inductive proof of Theorem \ref{t-mpos} 
using the fact that if $W(m,k)=\sum_{n=0}^\infty (n+k)^{n+m} x^n\!/n!$ then 
$dW(m,k)/dx = W(m+1,k+1)$.
\end{ex}

We can get convolution identities by applying \eqref{e-kforest} and \eqref{e-f1} 
to 
\begin{equation*}
\frac{F(x)^{k+l}}{1-T(x)} =F(x)^k \frac{F(x)^l}{1-T(x)}
\end{equation*}
and 
\begin{equation*}
F(x)^{k+l} =F(x)^k F(x)^l.
\end{equation*}
The first identity yields 
\begin{equation*}
(n+k+l)^n = \sum_{i=0}^n \binom ni  k(i+k)^{i-1}(n-i+l)^{n-i}.
\end{equation*}
and the second yields 
\begin{equation*}
(k+l)(n+k+l)^{n-1}=\sum_{i=0}^n \binom ni k(i+k)^{i-1} l (n-i+l)^{n-i-1}.
\end{equation*}
Note that these are identities of polynomials in $k$ and $l$.
If we set $k=x$ and $l=y-n$ in the first formula we get the nicer looking
\begin{equation}
\label{e-abel1}
(x+y)^n = \sum_{i=0}^n \binom ni x (x+i)^{i-1}(y-i)^{n-i}.
\end{equation}
Replacing $x$ with $x/z$ and $y$ with $y/z$ in \eqref{e-abel1}, and multiplying through by $z^n$ gives the homogeneous form
\begin{equation}
\label{e-abel2}
(x+y)^n = \sum_{i=0}^n \binom ni x (x+iz)^{i-1}(y-iz)^{n-i}
\end{equation}
which was proved by N. H.  Abel in 1826 \cite{MR1577607}. Note that for $z=0$, 
\eqref{e-abel2} reduces to the binomial theorem.
Riordan \cite[pp.~18--27]{MR0231725} gives a comprehensive account of Abel's identity and its generalizations, though he does not use Lagrange inversion.

\begin{ex}
(Chu \cite{MR2644861}.) Prove Abel's identity \eqref{e-abel1} using finite differences. 
(Start by expanding $(y-i)^{n-i}=[(x+y)-(x+i)]^{n-i}$ by the binomial theorem.)
\end{ex}

\subsection{Fuss-Catalan numbers}
\label{ss-fc}
The Fuss-Catalan (or Fu\ss-Catalan) numbers of order $p$ (also called generalized Catalan numbers) are the 
numbers $\frac{1}{pn+1}\binom{pn+1}{n} = \frac{1}{(p-1)n+1} \binom{pn}{n}$, which reduce to Catalan numbers for $p=2$. They were first studied by N. Fuss in 1791
\cite{fuss}. As we shall see, they are the 
coefficients of the power series $c_p(x)$, satisfying the functional equation  
\begin{equation}
\label{e-fc1}
c_p(x) = 1+ xc_p(x)^p,
\end{equation}
or equivalently,
\begin{equation*}
c_p(x) = \frac{1}{1-xc_p(x)^{p-1}},
\end{equation*}
as was shown using Lagrange inversion by Liouville \cite{Liouville}.

An account of these generating functions can be found Graham, Knuth, and Patashnik 
\cite[pp.~200--204]{MR1397498}.

It follows easily from \eqref{e-fc1} that 
\begin{gather}
c_p(x) -1 = xc_p(x)^p=\left(\frac{x}{(1+x)^p}\right)\inv,\label{e-cpinv}\\
xc_p(x)^{p-1}=\bigl(x(1-x)^{p-1}\bigr)\inv,\notag\\
xc_p(x^{p-1}) = (x-x^p)\inv,\notag
\end{gather}
and
\[c'_p(x) = \frac{c_p(x)^p}{1-pxc_p(x)^{p-1}}.\]

Lagrange inversion gives
\begin{equation}
\label{e-cp1}
c_p(x)^k =\sum_{n=0}^\infty \frac{k}{pn+k} \binom{pn+k}{n} x^n
\end{equation}
for all $k$.
With $R(t) = 1+xt^p$ we have $R'(t) = pxt^{p-1}$ so
\[1-R'(c_p(x)) = 1-pxc_p(x)^{p-1}=1-p(c_p(x) -1)/c_p(x)
  =1-p+pc_p(x)^{-1},\]
and thus by \eqref{e-r2},
\begin{equation}
\label{e-cp2}
\sum_{n=0}^\infty \binom{pn+k}{n} x^n=
\frac{c_p(x)^k}{1-pxc_p(x)^{p-1}}
=\frac{c_p(x)^{k+1}}{1-(p-1)(c_p(x) -1)}.
\end{equation}
Equivalently,
\begin{equation*}
\sum_{n=0}^\infty \binom{pn+k}{n} \left(\frac{x}{(1+x)^p}\right)^n 
  = \frac{(1+x)^{k+1}}{1-(p-1)x}\end{equation*}
and
\begin{equation*}
\sum_{n=0}^\infty \binom{pn+k}{n} \bigl(x(1-x)^{p-1}\bigl)^n
=\frac{1}{(1-px)(1-x)^k}.
\end{equation*}

The convolution identities obtained from \eqref{e-cp1} and \eqref{e-cp2}, known as Rothe-Hagen identities \cite{rothe, hagen, MR0087621} are
\begin{equation*}
\sum_{i+j=n} \frac{k}{pi+k}\binom{pi+k}{i}\cdot\frac{l}{pj+l}\binom{pj+l}{j}=\frac{k+l}{pn+k+l}\binom{pn+k+l}{n}.
\end{equation*}
and 
\begin{equation*}
\sum_{i+j=n} \frac{k}{pi+k}\binom{pi+k}{i}\binom{pj+l}{j}=\binom{pn+k+l}{n}.\end{equation*}

\begin{ex}
Show that $c_{-p}(x) = 1/c_{p+1}(-x)$.
\end{ex}

\begin{ex} Prove that $c_{p+q}(x)= c_p\bigl(xc_{p+q}(x)^q\bigr)$ 
(a) combinatorially
(b) algebraically 
(c) using Lagrange inversion.
\end{ex}

\begin{ex} 
Prove that 
\begin{equation*}
\bigl(x c_p(x^a)^b\bigr)\inv = x c_{ab-p+1}(-x^a)^b
\end{equation*}
(a) algebraically (b) using Lagrange inversion.
In particular, as noted by Dennis Stanton, if $f(x) = xc(x)^3$ then $f(x)\inv = -f(-x)$. 
\end{ex}

\begin{ex} (Mansour and Sun \cite[Example 5.6]{MR2433584}, Sun \cite{MR2644857}.)
Show that
\begin{equation*}
\frac{1}{1-x}c_3\left(\frac{x^2}{(1-x)^3}\right) = c_2(x).
\end{equation*}
\end{ex}

\begin{ex} Prove \eqref{e-cp1} and \eqref{e-cp2} by finite differences.
\end{ex}

\begin{ex} (Chu \cite{MR2644861}.) Prove the Hagen-Rothe identities by finite differences.
\end{ex}

Next, we prove Jensen's formula \cite{jensen}
\begin{equation}
\label{e-jensen}
\sum_{l=0}^n \binom{j+pl}{l}\binom{r-pl}{n-l}=\sum_{i=0}^n\binom{j+r-i}{n-i}p^i.
\end{equation}
By \eqref{e-cp2} we have
\begin{align*}
\sumz l \binom{pl+j}{l}x^l \sumz m \binom{pm+k}{m}x^m
 &= \frac{c_p(x)^{j+k}}{(1-pxc_p(x)^{p-1})^2}\\
 &= \frac{c_p(x)^{j+k}}{1-pxc_p(x)^{p-1}}\sum_{i=0}^\infty p^i x^i c_p(x)^{\p i}\\
 &=\sumz i \frac{p^i x^i c_p(x)^{j+k+\p i}}{1-pxc_p(x)^{p-1}}\\
 &=\sumz i p^i x^i \sum_{m=0}^\infty\binom{pm + j+k +\p i}{m}x^m\\
 &=\sumz n x^n \sum_{i=0}^n \binom{pn+j+k-i}{n-i}p^i.
\end{align*}
Equating coefficients of $x^n$ on both sides gives
\begin{equation*}
\sum_{l=0}^n \binom{pl+j}{l}\binom{p(n-l)+k}{n-l}=\sum_{i=0}^n\binom{pn+j+k-i}{n-i}p^i.
\end{equation*}
Setting $k=r-pn$ gives \eqref{e-jensen}.

We also have analogues of Theorems \ref{t-mneg} and \ref{t-mpos} for Fuss-Catalan numbers.
\begin{thm}
\label{t-fcneg}
Let $i$ and $j$ be nonnegative  integers with $i<j$.
Then 
\begin{equation*}
\sum_{n=0}^\infty \frac{(pn+i)!}{n!\, ((p-1)n+j)!} \frac{x^n}{(1+x)^{pn+i+1}}
\end{equation*}
is a polynomial $u_{i,j}(x)$ in $x$ of degree $j-i-1$.
\end{thm}

\begin{proof}
We have
\begin{align*}
\sum_{n=0}^\infty \frac{(pn+i)!}{n!\, ((p-1)n+j)!}& \frac{x^n}{(1+x)^{pn+i+1}}\\
  &=\sum_{n=0}^\infty \frac{(pn+i)!}{n!\, ((p-1)n+j)!} x^n 
  \sum_{l=0}^\infty (-1)^l \binom{pn+i+l}{l} x^l\\
  &=\sum_{m=0}^\infty x^m \sum_{n=0}^m \frac{(pn+i)!}{n!\, ((p-1)n+j)!}
    (-1)^{m-n} \binom{\p n+i+m}{m-n}.
\end{align*}
For $m\ge j-i$, the coefficient of $x^m$ may be rearranged to
\begin{equation*}
\frac{\bigl(m-(j-i)\bigr)!}{m!}\sum_{n=0}^m (-1)^{m-n}\binom{m}{n} \binom{\p n+i+m}{m-(j-i)}.
\end{equation*}
The sum is the $m$th difference of a polynomial of degree less than $m$ and is therefore 0. For $m=j-i-1$, the coefficient of $x^{j-i-1}$ reduces to 
\begin{equation*}
\frac{1}{m!}\sum_{n=0}^{m}(-1)^{m-n}\binom{m}{n} \frac{1}{\p n+j},
\end{equation*}
which is nonzero by the well-known identity
\begin{equation*}
\sum_{n=0}^m (-1)^n \binom mn \frac{a}{n+a} = \binom{m+a}{a}^{-1},
\end{equation*}
so the degree of the polynomial is not less than $j-i-1$.
\end{proof}
The first few values of these polynomials are 
\begin{align*}
u_{i,i+1}(x) &= \frac{1}{i+1}\\
u_{i,i+2}(x) &=\frac{1}{(i+1)(i+2)} -\frac{p-1}{(i+2)(p+i+1)}x\\
u_{i,i+3}(x) &= \frac{1}{(i+1)(i+2)(i+3)} -\frac{(p-1)(p+2i+4)}{(i+2)(i+3)(p+i+1)(p+i+2)}x\\
  &\quad+\frac{(p-1)^2}{(i+3)(p+i+2)(2p+i+1)}x^2
\end{align*}

As a simple example of Theorem \ref{t-fcneg}, the number of 2-stack-sortable permutations of $\{1,2,\dots,n\}$ is 
\[a_n=2\frac{(3n)!}{(n+1)!\,(2n+1)!}= 4\frac{(3n)!}{n!\,(2n+2)!}\]
(see \cite[Sequence A000139]{oeis}),
so by Theorem \ref{t-fcneg}, with $p=3$, $i=0$, and $j=2$, 
\begin{math}
\sumz n a_n {x^n}/(1+x)^{3n+1}
\end{math}
is a polynomial of degree 1, which is easily computed to be $2-x$.
Then by  \eqref{e-cpinv}, we find that 
\[\sumz n a_n x^n = 3c_3(x) -  c_3(x)^2,\] which can  be checked directly from 
\eqref{e-cp1}.

There is a  result  similar to Theorem  \ref{t-fcneg} for $i\ge j$, which we state without proof.
\begin{thm}
\label{t-ij}
Let $i$ and $j$ be nonnegative integers with $i\ge j$. Then 
\begin{equation*}
(1-\p x)^{2(i-j)+1}\sum_{n=0}^\infty \frac{(pn+i)!}{n!\, ((p-1)n+j)!} \frac{x^n}{(1+x)^{pn+i+1}}
\end{equation*}
is a polynomial in $x$ of degree at most $i-j$.
\qed
\end{thm}

\begin{ex} 
Prove Theorem \ref{t-ij}.
\end{ex}

\subsection{Narayana and Fuss-Narayana numbers}

The Narayana numbers may be defined by $N(n,i) = \frac{1}{n} \binom ni \binom n{i-1}$ for $n\ge1$. 
They have many combinatorial interpretations, in terms of Dyck paths, ordered trees, binary trees, and noncrossing partitions.

It is not hard to see from the formula for Narayana numbers that $N(n,i) = N(n, n+1-i)$. A generating function for the Narayana that exhibits this symmetry is given by the solution to the equation
\begin{equation}
\label{e-nara}
f = (1+xf)(1+yf).
\end{equation}
Lagrange inversion gives
\begin{align*}
f^k &= \sum_{n=0}^\infty \frac{k}{n}\cf{t^{n-k}} (1+xt)^n(1+yt)^n\\
  &=\sum_{n=0}^\infty \frac{k}{n}\cf{t^{n-k}}
     \sum_{i,j}\binom ni \binom nj x^i y^j t^{i+j}\\
  &=\sum_{i,j=0}^\infty \frac{k}{i+j+k}\binom{i+j+k}i\binom{i+j+k}j x^i y^j.
\end{align*}
In particular
\begin{align*}
f&=\sum_{i,j=0}^\infty \frac{1}{i+j+1}\binom{i+j+1}i\binom{i+j+1}{i+1} x^i y^j\\
  &=\sum_{n=1}^\infty \sum_{i=0}^{n-1} N(n,i+1) x^i y^{n-i-1}.
\end{align*}
Equation \eqref{e-nara} can be solved explicitly to give
\begin{equation*}
f=\frac{1-x-y-\sqrt{(1-x-y)^2 - 4xy}}{2xy}.
\end{equation*}

\begin{ex} Prove that
\begin{equation*}\frac{(1+x f)^r(1+y f)^s}{\sqrt{(1-x-y)^2-4xy}}
  =\sum_{i,j}\binom{r+i+j}{ i}\binom{s+i+j}{ j}x^iy^j
\end{equation*}
and
\begin{align*}
(1+x f)^r(1+y f)^s &=\sum_{i,j}\left[\binom{r+i+j-1}{ i}
    \binom{s+i+j}{ j}-\binom{r+i+j}{ i}\binom{s+i+j-1}{ j-1}
    \right]x^iy^j\\
  &=\sum_{i,j}\frac{rs+ri+sj}{(r+i+j)(s+i+j)}\binom{r+i+j}{ i}
    \binom{s+i+j}{ j}x^iy^j.
\end{align*}
The first formula is equivalent to a well-known generating function for Jacobi polynomials; see Carlitz \cite{MR0222357}.
\end{ex}

We may generalize \eqref{e-nara} to 
\begin{equation}
\label{e-mnar1}
f= (1+x_1f)^{r_1}(1+x_2f)^{r_2}\cdots (1+x_m f)^{r_m},
\end{equation}
for which Lagrange inversion gives
\begin{equation}
\label{e-mnar2}
f^k = \sum_{n=k}^\infty \sum_{i_1+\cdots+i_m = n-k}
  \frac{k}{n}\binom{r_1n}{i_1}\cdots \binom{r_m n}{i_m} x_1^{i_1}\cdots x_m^{i_m}.
\end{equation}
These numbers reduce to Catalan numbers for $k=m=1, r_1=2$ and to Narayana numbers for $k=1, m=2, r_1=r_2=1$. For $k=1, m=2, r_1=1$ they are sometimes called Fuss-Narayana numbers; see  Armstrong \cite{MR2561274}, Cigler \cite{MR919877}, Edelman \cite{Edelman}, Eu and Fu \cite{Eu-Fu}, and Wang \cite{wang}.  
The numbers for  $k=1$, $r_i=1$ for all $i$ have been 
called generalized Fuss-Narayana numbers by Lenczewski and Sa{\l}apata \cite{MR3084583}; they have also been studied by 
Edelman \cite{Edelman}, Stanley \cite{MR1444167}, and Xu \cite{dxu}.
The case $k=1, m=3, r_2=-r_1, r_3=1$ of these numbers was considered by Krattenthaler, \cite[equation (31)]{MR2249265}. We note that if $r_1=\cdots=r_m$, then $f$ is a symmetric function of $x_1, \dots, x_m$, and this symmetric function arises in the study of algebraic aspects of parking functions \cite{MR1444167}.

If we set $r_i=-s_i$ in \eqref{e-mnar1} and replace $x_i$ with $-x_i$, and $f$ with $g$, then \eqref{e-mnar1} becomes
\begin{equation}
\label{e-mnar3}
g= \frac{1}{(1-x_1g)^{s_1}(1-x_2g)^{s_2}\cdots (1-x_m g)^{s_m}},
\end{equation}
and, with the formula $\binom{-a}{i} = (-1)^{i}\binom{a+i-1}{i}$, \eqref{e-mnar2} becomes
\begin{equation}
\label{e-mnar4}
g^k = \sum_{n=k}^\infty \sum_{i_1+\cdots+i_m = n-k}
  \frac{k}{n}\binom{s_1n+i_1-1}{i_1}\cdots \binom{s_m n+i_m-1}{i_m} x_1^{i_1}\cdots x_m^{i_m}. 
\end{equation}
These numbers reduce to Catalan numbers for $k=m=s_1=1$. For $s_1=\cdots =s_m=1$ they have been considered by Aval \cite{MR2438172} and (for $k=1$) by Stanley \cite{MR1444167}.

Of special interest are the cases of \eqref{e-mnar1} that reduce to a quadratic equation, since in these cases there are simple explicit formulas for $f$.
If we take $m=1$, $r_1=r_2=1$, and $r_3=-1$ then with a change of variable names and one sign we have
\begin{equation*}
f=\frac{(1+xf)(1+yf)}{1-zf},
\end{equation*}
with the solution 
\[f = \frac{1-x-y -\sqrt{(1-x-y)^2 - 4xy - 4z}}{2(xy+z)},\]
and \eqref{e-mnar2} gives
\begin{equation*}
f^k = \sum_{n=k}^\infty \sum_{i_1+i_2+i_3= n-k}
  \frac{k}{n}\binom{n}{i_1}\binom{n}{i_2}\binom{n+i_3-1}{i_3} x^{i_1} y^{i_2}z^{i_3}.
\end{equation*}
Another case of \eqref{e-mnar1} that reduces to a quadratic is 
$f = \sqrt{(1+xf)(1+yf)}$; we leave the details to the reader.

\subsection{Raney's equation}
\label{s-raney}

G. Raney \cite{MR0166114} considered the equation 
\begin{equation}
\label{e-raney1}
f = \sum_{m=1}^\infty A_m e^{B_m f},
\end{equation}
in which $f$ is a  power series in the indeterminates $A_m$ and $B_m$. He used Pr\"ufer's correspondence to give a combinatorial derivation of a formula for the coefficients in this power series.

We can use Lagrange inversion to give a formula for the coefficients of $f$.
\begin{thm}
Let $f$ be the  power series in $A_m$ and $B_m$ satisfying \eqref{e-raney1}, and let $k$ be a positive integer. Let $i_1, i_2, \dots$ and $j_1, j_2, \dots$ be nonnegative integers, only finitely many of which are nonzero. If $i_1+i_2+\dots =k+j_1+j_2+\cdots$ then the coefficient of $A_1^{i_1}A_2^{i_2}\cdots B_1^{j_1}B_2^{j_2}\cdots$ in $f^k$ is 
\begin{equation*}
k\frac{(i_1+i_2+\cdots -1)!}{i_1!\, i_2!\cdots} 
\frac{i_1^{j_1}}{j_1!}\frac{i_2^{j_2}}{j_2!}\cdots
\end{equation*}
and if  $i_1+i_2+\dots \ne k+j_1+j_2+\cdots$ then the coefficient is zero.

\end{thm}

\begin{proof}
Applying equation \eqref{e-r00} to \eqref{e-raney1} gives
\begin{align*}
f^k &= \sum_{n=k}^\infty \frac{k}{n}\cf{t^{n-k}} \biggl(\sum_m A_m e^{B_m t}\biggr)^n\\
    &= \sum_{n=k}^\infty \frac{k}{n}\cf{t^{n-k}} 
    \sum_{i_1+i_2+\cdots = n}\frac{n!}{i_1!\,i_2!\cdots} A_1^{i_1}A_2^{i_2}\cdots
    e^{i_1B_1 t}e^{i_2B_2 t}\cdots\\
    &= \sum_{n=k}^\infty \frac{k}{n}\cf{t^{n-k}} 
    \sum_{i_1+i_2+\cdots = n}\frac{n!}{i_1!\,i_2!\cdots} A_1^{i_1}A_2^{i_2}\cdots
   \sum_{j_1} \frac{(i_1B_1t)^{j_1}}{j_1!}
   \sum_{j_2} \frac{(i_2B_2t)^{j_2}}{j_2!}\cdots\\
   &=\sum_{n=k}^\infty
   \sum_{\substack{i_1+i_2+\cdots = n\\
   j_1+j_2+\cdots = n-k}}
   k \frac{(n-1)!}{i_1!\,i_2!\cdots}
   A_1^{i_1}A_2^{i_2}\cdots B_1^{j_1}B_2^{j_2}\cdots\frac{i_1^{j_1}}{j_1!}\frac{i_2^{j_2}}{j_2!}\cdots,
\end{align*}
and the formula follows.
\end{proof}

A combinatorial derivation of Raney's formula has also been given by D. Knuth \cite[Section 2.3.4.4]{knuth1}.

\section{Proofs}
In this section we give several proofs of the Lagrange inversion formula.
\label{s-proofs}
\subsection{Residues}
The simplest proof of Lagrange inversion is due to Jacobi \cite{Jacobi}. We define the \emph{residue} $\res f(x)$ of a  Laurent series $f(x) = \sum_{n}f_n x^n$ to be $f_{-1}$.

Jacobi proved the following change of variables formula for residues:

\begin{thm}
\label{t-res}
 Let $f$ be a  Laurent series and let $g(x)=\sum_{n=1}^\infty g_n x^n$ be a  power series with $g_1\ne0$. Then
\begin{equation*}
\res f(x) = \res f(g(x))g'(x).
\end{equation*}
\end{thm}
\begin{proof}
By linearity, it is sufficient to prove the formula when $f(x) = x^k$ for some integer $k$. If $k\ne -1$ then $\res x^k = 0$ and 
\begin{equation*}
\res g(x)^k g'(x) = \res \frac{d\ }{dx} g(x)^{k+1}/(k+1) = 0,
\end{equation*}
since the residue of a derivative is 0.

If $k=-1$ then $\res x^k = 1$ and 
\begin{align*}
\res g(x)^k g'(x) &= \res g'(x)/g(x) = \res \frac{g_1+2g_2x+\cdots}{g_1x+g_2x^2+\cdots}\\
&=\res \frac1x\cdot \frac{g_1+2g_2x+\cdots}{g_1+g_2x+\cdots}=1.\qedhere
\end{align*}
\end{proof}

Jacobi's paper \cite{Jacobi} contains a multivariable generalization of Theorem \ref{t-res}; see also Gessel \cite{MR894817} and Xin \cite{MR2164920}.

Now let $f(x)$ and $g(x)$ be compositional inverses. Then for any Laurent series $\phi$,
\begin{equation*}
\cf{x^n} \phi(f) = \res \frac{\phi(f)}{x^{n+1}}
  = \res\frac{\phi(f(g))g'}{g^{n+1}} =\res\frac{\phi(x)g'}{g^{n+1}}.
\end{equation*}
This is equation \eqref{e-bg}, which we have already seen is equivalent to the other forms of Lagrange inversion.

\begin{ex} Let $f$ be a  Laurent series and let $g(x)=\sum_{n=m}^\infty g_n x^n$ be a  Laurent series with $g_m\ne0$. Show that if $f(g(x))$ is well-defined as a  Laurent series then 
\begin{equation*}
m\res f =\res f(g(x)) g'(x).
\end{equation*}
\end{ex}

\begin{ex}
\label{ex-Todd} (Hirzebruch \cite{MR0368023}; see also Kneezel \cite{60478}.) 
Use the change of variables formula (Theorem \ref{t-res}) to show that the unique power series $f(x)$ satisfying 
\begin{equation*}
\res \left(\frac{f(x)}{x}\right)^n=1
\end{equation*}
for all $n\ge1$ is $f(x)=x/(1-e^{-x})$.
\end{ex}

\subsection{Induction}\ 
In this proof and the next we consider the equation $f=xR(f)$, where $R(t)$ is a power series. If $R(t)$ has no constant term then $f=0$ and the formulas are trivial. So we may assume that $R(t)$ has a nonzero constant term, and thus $f$ exists and is unique, since it is the compositional inverse of $x/R(x)$.

We now give an inductive proof of \eqref{e-b}: for any power series $\phi(t)$,
\begin{equation}
\label{e-i1}
\cf{x^n}\phi(f)=\cf{t^n}\left(1-\frac{tR'(t)}{ R(t)}\right)\phi(t)R(t)^n.
\end{equation}
(As noted in section 
\ref{s-formula}, this implies that \eqref{e-i1} holds more generally when $\phi(t)$ is a Laurent series.)

We first take care of the case  in which $\phi(t)=1$. The case $\phi(t)=1$,
$n=0$ is trivial. If 
$\phi(t)=1$ and $n>0$, we have
\begin{align}
\left(1-\frac{tR'(t)}{ R(t)}\right)R(t)^n
  &= R(t)^n-tR'(t)R(t)^{n-1}\notag\\
  &=R(t)^n-\frac{t}{ n}\frac{d\ }{ dt}R(t)^n.\label{e-i2}
\end{align}
Now for any power series $u(t)$, $$\cf{t^n}\left(u(t)-\frac{t}{ n}u'(t)\right)=0.$$
With\eqref{e-i2}, this proves \eqref{e-i1} for $\phi(t)=1$, $n>0$.

Now we prove the formula \eqref{e-i1} by induction on $n$. It is clear that \eqref{e-i1}holds for
$n=0$. Now let us suppose that for some nonnegative integer $m$, \eqref{e-i1} holds 
for all $\phi$ when $n=m$. We now want to show that \eqref{e-i1} holds for all $\phi$
when $n=m+1$. By linearity and the case $\phi(t)=1$, it is enough to prove \eqref{e-i1}
for $n=m+1$ and $\phi(t)=t^k$, where $k\ge 1$.
In this case we have
\begin{align*}
\cf{x^{m+1}}f^k &=\cf{x^{m+1}}f^{k-1}\cdot xR(f)\\
  &=\cf{x^m}f^{k-1}R(f)\\
  &=\cf{t^m}\left(1-\frac{tR'(t)}{ R(t)}\right)t^{k-1}R(t)R(t)^m\\
  &=\cf{t^{m+1}}\left(1-\frac{tR'(t)}{ R(t)}\right)t^kR(t)^{m+1}.
\end{align*}

\subsection{Factorization}
Another proof is based on  a version of the ``factor theorem": if $f = xR(f)$ then $t-f$ divides $t-xR(t)$.  
This proof is taken from Gessel \cite{MR570213} but it is similar to Lagrange's original proof \cite{lagrange}.

We will prove \eqref{e-pa}, which gives a formula for $f^k$ where $f=f(x)$ satisfies
$f=xR(f)$. 

First we recall Taylor's theorem for power series: if $P(t)$ is a power series in $t$, and $\alpha$ is an element of the coefficient ring, then 
\begin{equation}
\label{e-taylor}
P(t) = \sumz n \frac{(t-\alpha)^n}{n!} P\psup n(\alpha),
\end{equation}
as long as this sum  is summable. 
(The case $P(t)= t^m$ of \eqref{e-taylor} is just the binomial theorem, and the general case follows by linearity.)

Now let us apply \eqref{e-taylor} with $P(t) = t-xR(t)$ and $\alpha=f$, where  $f = xR(f)$ so that $P(f)=0$.
Then we have 
\begin{align}
t-xR(t) &= 0 + (t-f)\bigl(1-xR'(f)\bigr) + (t-f)^2 S(x,t) \notag\\
  &= (t-f) Q(x,t), \label{e-factor1}
\end{align}
where $S(x,t)$ is a power series and  $Q(x,t)$ is a power series with constant term 1.

Equation \eqref{e-factor1} is an identity in the ring $C[[x,t]]$, which is naturally embedded in the  ring $C((t))[[x]]$ of  power series in $x$ with coefficients that are Laurent series in $t$. In this ring, series like $\sum_{n=0}^\infty (x/t)^n$ are allowed, even though they have infinitely many negative powers of $t$, since the coefficient of any power of $x$ is a Laurent series in $t$. 
We now do some computations in $C((t))[[x]]$.

By \eqref{e-factor1}, we have
\begin{math}
1-x R(t)/t =  (1-f/t)Q(x,t).
\end{math}
Since $xR(t)/t$ and $f/t$ are divisible by $x$ and $Q(x,t)$ is a power series in $x$ and $t$ with constant term 1, we may take logarithms to obtain
\begin{equation}
\label{e-RQ}
-\log(1- x  R(t)/t)=-\log( 1-f/t )-\log Q(x,t).
\end{equation}
Note that $\log Q(x,t)$ is  a power series in $x$ and $t$, and so has no negative powers of $t$.  
Now we equate coefficients of $x^nt^{-k}$ on both sides of \eqref{e-RQ} where $n$ and $k$ are both positive integers.    
On the left we have
\begin{math}
\cf{t^{-k}} (R(t)/t)^n/n =  \cf{t^{n-k}} R(t)^n/n
\end{math}
and on the right we have 
\begin{math}
\cf{x^n} f^k/k.
\end{math}
Thus, 
\begin{equation*}
\cf{x^n} f^k = \frac kn\cf{t^{n-k}} R(t)^n,
\end{equation*}
which is \eqref{e-pa}.

\begin{ex} (Gessel \cite{MR570213}.) Derive \eqref{e-pd} similarly.\end{ex}

\subsection{Combinatorial proofs}
There are several different combinatorial proofs of Lagrange inversion. They all interpret the solution $f$ of  $f=xR(f)$  or $f=R(f)$  as counting certain trees. Here $f$ may be interpreted as either an ordinary or exponential generating function and thus different types of trees may be involved. 

In ordinary generating function proofs, $f$ will count unlabeled \emph{ordered trees} (also called \emph{plane trees}), which are rooted trees in which the children of each vertex are linearly ordered. (See Figure~\ref{f-otree}.) 
\begin{figure}[htbp] 
   \centering
   \includegraphics[width=1.2in]{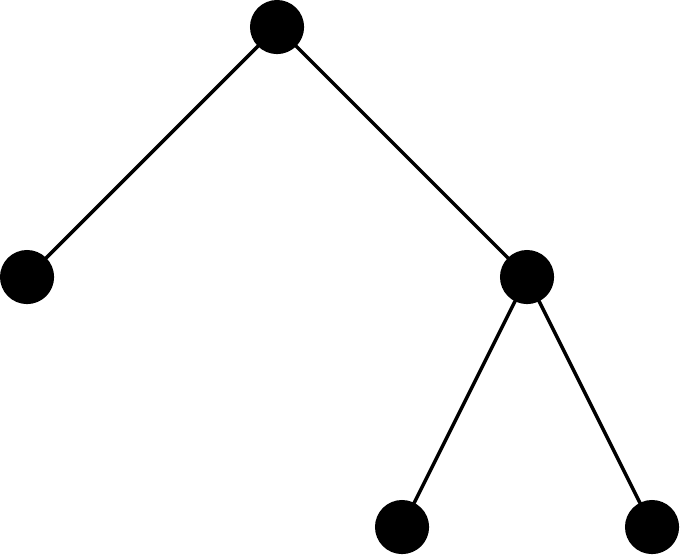} 
   \caption{An ordered tree}
   \label{f-otree}
\end{figure}
More generally, $f^k$ will count $k$-tuples of ordered trees,
which we also  call \emph{forests of ordered trees}.

If $f=xR(f)$, where $R(t) = \sum_{i=0}^\infty r_i t^i$, then the coefficient of $x^n$ in $f$ is the sum of the weights of the ordered trees with a total of $n$ vertices, where 
the weight of a tree is the product of the weights of its vertices and the weight of a vertex with $i$ children is  $r_i$.
(For example, the tree of Figure \ref{f-otree} has weight $r_0^3r_2^2$.) So to give a combinatorial proof of formula 
\eqref{e-pa},
we show that the sum of the weights of all  $k$-tuples of ordered trees with a total of $n$ vertices is $(k/n)\cf{t^{n-k}}R(t)^n$, or equivalently by \eqref{e-lagr2}, the number of $k$-tuples of ordered trees in which  $n_i$ vertices have $i$ children for each $i$ is equal to \begin{equation}
\label{e-forest1}
\frac{k}{n}\binom{n}{n_0,n_1,n_2,\dots}
\end{equation}
if 
\begin{equation}
\label{e-condition}
n=\sum_k n_i \quad \text{and}\quad  n-k = \sum_k in_i,
\end{equation}
and  is zero otherwise.
W. T. Tutte \cite{tutte} gave the case $k=1$ of \eqref{e-forest1}, which he derived (in a roundabout way) from Lagrange inversion.

We can also work with exponential generating functions. One way to do this is to consider the equation 
\begin{math}
f=x\sum_{i=0}^\infty s_i f^i\!/i!
\end{math}
where $x$ is the exponential variable and the $s_n$ are weights.
Then by the properties of exponential generating functions (see, e.g., \cite[Chapter 5]{MR1676282}) $f$ counts labeled rooted trees where a vertex with $i$ children is weighted $s_i$. More precisely, the coefficient of $x^n\!/n!$ in $f^k/k!$ is the sum of the weights of all forests of $k$ rooted trees with vertex set $[n]=\{1,2,\dots,n\}$.  Then Lagrange inversion in the form \eqref{e-lagr2}  is equivalent to assertion that the number of forests of $k$ rooted trees with vertex set $[n]$ in which  $n_i$ vertices have $i$ children 
is equal to 
\begin{equation}
\label{e-ltree}
\frac{(n-1)!}{(k-1)!\, 0!^{n_0}1!^{n_1}2!^{n_2}\cdots}\binom{n}{n_0,n_1,\dots}
\end{equation}
with the same conditions on $n_0,n_1,\dots $ as before. (See Stanley \cite[p.~30, Corollary 5.3.5]{MR1676282}.)

Another exponential generating function approach is to consider the equation \[f=\sum_{i=0}^\infty s_i \dpow fi\] where we work with exponential generating functions in the variables $s_0,s_1,\dots$. Then by the properties of multivariable exponential generating functions, the coefficient of 
\begin{equation*}
\dpow{s_0}{n_0}\dpow{s_1}{n_1}\dots
\end{equation*}
in $f^k/k!$ is the number of forests of $k$ rooted trees in  which for each $i$, $n_i$ vertices have $i$ children and these vertices are labeled $1,2,\dots, n_i$. (Since vertices with the same label have different numbers of children, there are no nontrivial label-preserving automorphisms of these forests.) For example, such a forest with $k=2$, $n_0=5$, $n_2=3$, and $n_i=0$ for $i\notin\{0,2\}$ is
show in Figure \ref{f-forest2}.
\begin{figure}[htbp] 
   \centering
   \includegraphics[width=2in]{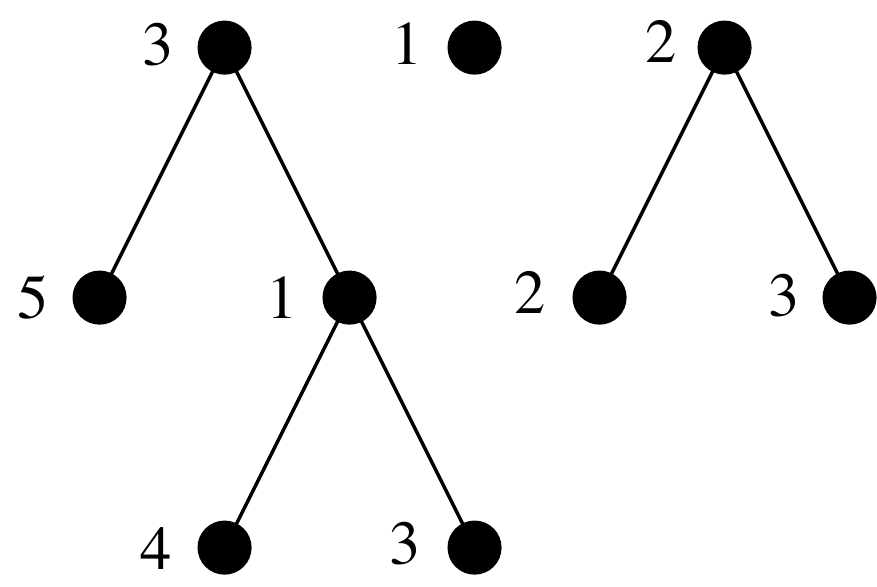} 
   \caption{A forest with $n_0=5$, $n_2=3$}
   \label{f-forest2}
\end{figure}
Then Lagrange inversion in the form \eqref{e-lagr2} (with $x=1$) is equivalent to the assertion that the number of such forests is 
\begin{equation}
\label{e-rexp}
\frac{(n-1)!}{(k-1)!\,0!^{n_0}1!^{n_1}2!^{n_2}}\cdots,
\end{equation}
where $n=n_0+n_1+\cdots$ and $n_1+2n_2+3n_3+\cdots = n-k$.

\subsection{Raney's proof}
The earliest combinatorial proof of Lagrange inversion is that of Raney~\cite{MR0114765}. We sketch here a proof that is based on Raney's though the details are different. (See also Stanley \cite[pp.~31--35 and 39--40]{MR1676282}.) We define the \emph{suffix code} $c(T)$ for an ordered tree $T$ to be a sequence of nonnegative integers defined recursively: If the root $r$ of $T$ has $j$ children, and the trees rooted at the children of $r$ are $T_1,\dots, T_j$, then $c(T)$ is the concatenation $c(T_1) \cdots c(T_j) j$. More generally, the suffix code for a $k$-tuple $(T_1,\dots, T_k)$ of ordered trees is the concatenation $c(T_1)\cdots c(T_k)$. We defined the \emph{reduced code} of a tree or $k$-tuple of trees to be the sequence obtained from the suffix code by subtracting 1 from each entry.

For example the suffix code of the $k$-tuple of trees in Figure  \ref{f-ordforest}
is $0\,0\,2\,0\,1$ and the reduced code is $\1 \, \1\, 1 \, \1\, 0$, where $\1$ denotes $-1$.
\begin{figure}[htbp] 
   \centering
   \includegraphics[width=1in]{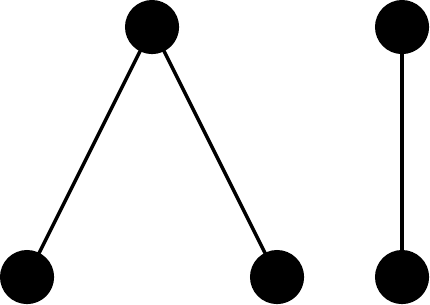} 
   \caption{A forest of ordered trees}
   \label{f-ordforest}
\end{figure}

The following lemma can be proved by induction:

\begin{lem}
\label{l-code}
\ 
\begin{enumerate}
\item[(i)] A forest of ordered trees is uniquely determined by its reduced code.
\item[(ii)] A sequence $a_1 a_2\cdots a_n$ of integers greater than or equal to $-1$ is the reduced code of an ordered $k$-forest if and only if $a_1+\cdots + a_n=-k$ and $a_1+\cdots+a_i$ is negative for $i=1,\dots, n$.
\qed
\end{enumerate}
\end{lem}

We also need a lemma, due to Raney, that generalizes  the ``cycle lemma" of Dvoretzky and Motzkin \cite{MR0021531}. It can be proved by induction, or in other ways. (See, e.g., Stanley \cite[pp.~32--33]{MR1676282}.)
\begin{lem}
\label{l-cycle}
Let $a_1\cdots a_n$ be a sequence of integers greater than or equal to $-1$ with sum $-k<0$. Then there are exactly $k$ integers $i$, with $1\le i\le n$, such that the sequence
$a_i\cdots a_n a_1\cdots a_{i-1}$ has all partial sums negative. \qed
\end{lem}

We can now prove Lagrange inversion. We want to prove that the sum of the weights of all $k$-forests with $n$ vertices is $k/n$ times 
\begin{equation*}
\cf{t^{n-k}} R(t)^{n}=\cf{t^{-k}}\left(\frac{R(t)}{t}\right)^n,\end{equation*}
where $R(t) = \sum_{n=0}^\infty s_n t^n$.
Let us define the weight of a sequence $a_1\cdots a_n$ of integers to be the product $s_{a_1+1}\cdots s_{a_n+1}$. Then by Lemma \ref{l-code}, the sum of the weights of all $k$-forests with $n$ vertices is the sum of the weights of all sequences of integers of length $n$, with entries greater than or equal to $-1$, with sum $-k$, and with all partial sums negative. It is clear that $\cf{t^{-k}}(R(t)/t)^n$ is the sum of the weights 
of all sequences of integers of length $n$, with entries greater than or equal to $-1$, and with sum $-k$. But by Lemma \ref{l-cycle}, a proportion $k/n$ of these sequences have all partial sums negative.

\subsection{Proofs by labeled trees}
We can derive Lagrange inversion by counting labeled trees with the following result, which seems to have first been proved by Moon \cite{MR0175115}.
\begin{thm}
\label{t-treemult}
Let $m$ be a positive integer and let $d_1,d_2,\dots, d_m$ be positive integers. Then the number of (unrooted) trees with vertex set $[m]$ in which vertex $i$ has degree $d_i$ is the multinomial coefficient
\begin{equation}
\label{e-multico}
\binom{m-2}{d_1-1,\dots, d_m-1}
\end{equation}
if $\sum_{i=1}^m d_i = 2(m-1)$ and is 0 otherwise. 
\end{thm}

We will prove Theorem \ref{t-treemult} a little later; we first  look at some of its consequences.
A corollary of Theorem \ref{t-treemult} allows us to count forests of rooted trees:
\begin{cor}
\label{c-forests}
Let $e_1,e_2,\dots, e_n$ be nonnegative integers with $e_1+e_2+\cdots+e_n = n-k$, and let $k$ be a positive integer. Then the number of forests of $k$ rooted trees, with vertex set $[n]$, in which vertex $i$ has $e_i$ children is 
\begin{equation*}
\binom{n-1}{k-1, e_1,\dots e_n}.
\end{equation*}
\end{cor}
\begin{proof}
Let $T$ be a tree on $[n+1]$ in which vertex $n+1$ has degree $k$,  and vertex $i$ has  degree $e_i+1$   for $1\le i\le n$ (so that $\sum_{i=1}^n e_i =n-k$). Removing vertex $n+1$ from $T$ and rooting the resulting component trees at the neighbors of  $n+1$ in $T$ gives a forest $F$ of $k$ rooted trees in which vertex $i$ has $e_i$ children, and this operation gives a bijection from trees on $[n+1]$ in which vertex $n+1$ has degree $k$ and vertex $i$ has degree $e_i+1$  for $1\le i \le n$ to the set of rooted forests of $k$ trees on $[n]$ in which vertex $i$ has $e_i$ children. The result then follows from Lemma \ref{t-treemult}.
\end{proof}

It follows from Corollary  \ref{c-forests} that the number of
forests of $k$ rooted trees in  which for each $i$, $n_i$ vertices have $i$ children and these vertices are labeled $1,2,\dots, n_i$ is given by \eqref{e-rexp}, which as we saw earlier, is equivalent to Lagrange inversion. 

Now to count $k$-forests on $[n]$ in which $n_i$ vertices have $i$ children, we first assign the number of children to each element of $[n]$, which can be done in $\binom{n}{n_0,n_1,\dots}$ ways. For each assignment, the number of trees is 
\begin{equation*}
\frac{(n-1)!}{(k-1)!\,0!^{n_0}1!^{n_1}\cdots}
\end{equation*}
Multiplying these factors gives \eqref{e-ltree}.

We now present sketches of three proofs of Theorem \ref{t-treemult}. They all depend on the fact that a tree with at least two vertices must have at least one leaf (vertex of degree 1).

The first proof uses Pr\"ufer's correspondence \cite{prufer}. Given a tree $T_1$ with vertex set $[n]$, let $a_1$ be the least leaf of $T_1$ and let $b_1$ be the unique neighbor of $a_1$ in $T_1$. Let $T_2$ be the result of removing  $a_1$ and its incident edge from $T_1$. Let $a_2$ be the least leaf of $T_2$ and let $b_2$ be its neighbor in $T_2$. Continue in this way to define $b_3,\cdots, b_{n-2}$. Then the \emph{Pr\"ufer code} of $T_1$ is the sequence $(b_1, b_2, \cdots, b_{n-2})$. It can be shown that the map that takes a tree to its Pr\"ufer code is a bijection from the set of trees with vertex set $[n]$, for $n\ge 2$, to the set of sequences $b_1b_2\dots b_{n-2}$ of elements of $[n]$, with the property that a vertex of degree $d$ appears $d-1$ times in the Pr\"ufer code. Then, as noted by Moon 
\cite{MR0175115}, 
Theorem \ref{t-treemult} is an immediate consequence, since the multinomial coefficient counts Pr\"ufer codes of trees in which vertex $i$ has degree $d_i$.

The second proof is by induction, and is due to Moon \cite{MR0231755} (see also  \cite[p.~13]{MR0274333}). For $m\ge2$, let $T(m;d_1,\dots, d_m)$ be the number of trees on $[m]$ in which vertex $i$ has degree $d_i$ and let $U(m;d_1\dots, d_m)$ be the multinomial coefficient \eqref{e-multico}, which is 0 if any $d_i$ is less than 1 or if $\sum_i d_i\ne 2(n-1)$.

Clearly $T(m;d_1,\dots, d_m)$ is equal to $U(m;d_1\dots, d_m)$ for $m=2$. Now suppose that $n>2$ and that $T(m;d_1,\dots, d_m)$ is equal to $U(m;d_1\dots, d_m)$ for $m=n-1$ and all choices of $d_1,\dots, d_m$. 

We next observe that if $d_n=1$ then 
\begin{equation}
\label{e-multirec}
T(n; d_1,\dots, d_n) = \sum_{i=1}^{n-1} T(n-1; d_1,\dots, d_i-1,\dots, d_{n-1})
\end{equation}
since every tree in which $n$ is a leaf is obtained by joining vertex $n$ with an edge to some vertex of a tree on $[n-1]$,  increasing by 1 the degree of this edge and leaving the other degrees unchanged. Then by the inductive hypothesis and a well-known recurrence for multinomial coefficients, $T(n;d_1,\dots, d_n)$ is equal to $U(n;d_1\dots, d_n)$. We have assumed that $d_n=1$, but since $T(n;d_1,\dots,d_n)$ is symmetric in $d_1,\dots, d_n$ and every tree has at least one leaf, the result holds without this assumption.

R\'enyi \cite{MR0263708} (see also Moon \cite[p.~13]{MR0274333}) gave an elegant variation of this proof. He showed by induction that for $m\ge2$ the number of trees on $[m]$  with degree sequence $(d_1, \dots, d_m)$, i.e.,  trees in which vertex $i$ has degree $d_i$,  is the coefficient of $x_1^{d_1-1}\cdots x_m^{d_m-1}$ in $(x_1+\cdots+x_m)^{m-2}$. The result clearly holds for $m=2$, so suppose that $n>2$ and that the result holds when $m=n-1$. To show that it holds for $m=n$, we note that every tree on $[n]$ has at least one leaf, so by symmetry, we may assume, without loss of generality, that $d_n=1$. Thus we need only show that the number of trees on $[n]$ with degree sequence $(d_1,\dots, d_{n-1}, 1)$ is the coefficient of $x_1^{d_1-1}\cdots x_{n-1}^{d_{n-1}-1}$ in 
\begin{equation*}
(x_1+\cdots+x_{n-1})^{n-2} = (x_1+\cdots +x_{n-1})\cdot(x_1+\cdots+x_{n-1})^{n-3}.
\end{equation*}
But every tree on $[n]$ in which vertex $n$ is a leaf is obtained by joining vertex $n$ with an edge to some vertex of a tree on $[n-1]$,  increasing by 1 the degree of this vertex and leaving the other degrees unchanged. Thus by the induction hypothesis, the contribution from trees in which vertex $n$ is joined to vertex $i$ is $x_i(x_1+\cdots+x_{n-1})^{n-3}$ and the result follows.

\medskip
\small{
\textbf{Acknowledgment.} I would like to thank Sateesh Mane and an anonymous referee for helpful comments.}

\bibliography{Lagrange-refs}{}
\bibliographystyle{amsplain}

\end{document}